\documentclass[]{article}

\usepackage[utf8]{inputenc} 
\usepackage[T1]{fontenc} 
\usepackage[english]{babel} 
\usepackage{amsmath} 
\usepackage{amsthm}
\usepackage{amssymb}
\usepackage{nicefrac} 
\usepackage{graphicx}
\usepackage{braket}
\usepackage{subfig}
\usepackage{csquotes}
\usepackage{lineno,hyperref}
\usepackage[style=numeric,backend=biber]{biblatex}
\providecommand{\abs}[1]{\left|#1\right|}
\providecommand{\norm}[1]{\left \| #1\right \|}

\theoremstyle{plain}
\newtheorem{theorem}{Theorem}[]

\theoremstyle{plain}
\newtheorem{lemma}{Lemma}[]

\theoremstyle{plain}
\newtheorem{corollary}{Corollary}[]

\theoremstyle{plain}
\newtheorem{proposition}{Proposition}[]

\theoremstyle{remark}
\newtheorem{remark}{Remark}[]

\title{Existence of solutions to higher order Lane-Emden type systems}
\author{Delia Schiera \thanks{Università degli Studi dell’Insubria, Dipartimento di Scienza e Alta Tecnologia, via Valleggio 11, Como, Italy
\newline E-mail address: \texttt{d.schiera@uninsubria.it}}}
\date{}

\begin{document}
\maketitle

\begin{abstract}
We prove existence results for the Lane-Emden type system 
\[ \begin{cases}
\begin{aligned}
(-\Delta)^{\alpha} u=\abs{v}^q \\
(-\Delta)^{\beta} v= \abs{u}^p
\end{aligned} \text{ in } B_1 \subset \mathbb{R}^N \\
\frac{\partial^{r} u}{\partial \nu^{r}}=0, \, r=0, \dots, \alpha-1,  \text{ on } \partial B_1 \\
\frac{\partial^{r} v}{\partial \nu^{r}}=0, \, r=0, \dots, \beta-1, \text{ on } \partial B_1.
\end{cases}
 \]
where $B_1$ is the unitary ball in $\mathbb{R}^N$, $N >\max \{2\alpha, 2\beta \}$, $\nu$ is the outward pointing normal, $\alpha, \beta \in \mathbb{N}$, $\alpha, \beta \ge 1$ and $(-\Delta)^{\alpha}= -\Delta((-\Delta)^{\alpha-1})$ is the polyharmonic operator.  A continuation method together with a priori estimates will be exploited. Moreover, we prove uniqueness for the particular case $\alpha=2$, $\beta=1$ and $p, q>1$. 
\end{abstract}
\\
\begin{MScodes}
35J48, 35J58. 
\end{MScodes}


\section{Introduction}
We will be concerned with the study of existence of solutions to the following Lane-Emden type system (\cite{deFigueiredoFelmer94_2, HulshofVanderVorst93, Mitidieri93})
\begin{equation}\label{LE_D}
\begin{cases}
\begin{aligned}
(-\Delta)^{\alpha} u=\abs{v}^{q-1} v \\
(-\Delta)^{\beta} v= \abs{u}^{p-1} u
\end{aligned} \quad \text{ in } \Omega \subset \mathbb{R}^N \\
\frac{\partial^{r} u}{\partial \nu^{r}}=0, \, r=0, \dots, \alpha-1, \text{ on } \partial \Omega \\
\frac{\partial^{r} v}{\partial \nu^{r}}=0, \, r=0, \dots, \beta-1, \text{ on } \partial \Omega
\end{cases}
\end{equation}
where $\Omega$ is a smooth bounded domain, $N >\max \{2\alpha, 2\beta \}$, $\nu$ is the outward pointing normal, $\alpha, \beta \in \mathbb{N}$, $\alpha, \beta \ge 1$ and $(-\Delta)^{\alpha}= -\Delta((-\Delta)^{\alpha-1})$ is the polyharmonic operator.
On the one hand, systems in which two nonlinear PDE are coupled in a Hamiltonian fashion, namely
\[
\begin{cases}
L_1 u = \frac{\partial H}{\partial v} (u, v) \\
L_2 v= \frac{\partial H}{\partial u} (u, v)
\end{cases}
\]
with given function $H \colon \mathbb{R}^2 \to \mathbb{R}$ and operators $L_1, L_2$, have attracted a lot of attention in the last two decades both from the Mathematical as well as Physical point of view, as those models describe, among many others, nonlinear interaction between fields, see \cite{Benci14, Yang01}. On the other hand, the polyharmonic operator appears in many different contexts, such as in the modeling of classical elasticity problems (in particular suspension bridges \cite{GazzolaGrunauSweers10}), as well as Micro Electro-Mechanical Systems (MEMS), see \cite{CassaniKaltenbacherLorenzi09} and references therein. 

\noindent In order to present our results, let us first briefly survey some existing literature about Lane-Emden type systems. 
Consider 
\begin{equation}\label{LaneEmden}
\begin{cases}
\begin{aligned}
-\Delta u &= \abs{v}^{q-1}v \\
-\Delta v &= \abs{u}^{p-1} u
\end{aligned} & \Omega \subset \mathbb{R^N}\\
u=v=0 & \partial \Omega
\end{cases}
\end{equation}
where $N >2$ and $\Omega$ is a smooth bounded domain, which is \eqref{LE_D} in the particular case $\alpha=\beta=1$. 
This problem turns out to be variational, namely weak solutions to \eqref{LaneEmden} are critical points of the functional
\begin{equation}\label{I2} I(u,v)=\int \nabla u \nabla v - \frac{1}{p+1} \int \abs{u}^{p+1} - \frac{1}{q+1} \int \abs{v}^{q+1}.  \end{equation}
In the case $u=v$, $p=q$, \eqref{LaneEmden} reduces to the single equation
\begin{equation}\label{eqnLE} 
\begin{cases}
-\Delta u = \abs{u}^{p-1} u &\Omega \subset \mathbb{R^N}\\
u=0 & \partial \Omega
\end{cases}
\end{equation}
and the corresponding functional is given by 
\begin{equation}\label{I} I(u)=\frac{1}{2} \int \abs{\nabla u}^2 - \frac{1}{p+1} \int \abs{u}^{p+1}. \end{equation}
The existence of weak solutions to  \eqref{eqnLE} can be proved by exploiting the Mountain Pass Theorem \cite{AmbrosettiRabinowitz73} if $p \in (1, \frac{N+2}{N-2})$, whereas if $p \ge \frac{N+2}{N-2}$ then no positive solutions do exist, due to the Pohozaev identity \cite{Pohozaev65}. Therefore, the value $\frac{N+2}{N-2}$ is the threshold between existence and non existence of solutions to \eqref{eqnLE}. Note that $\frac{N+2}{N-2} = p^* - 1$ where $p^*$ is the critical Sobolev exponent. 

\noindent When considering the case of systems of the form \eqref{LaneEmden}, the situation changes deeply, since the quadratic part of the functional \eqref{I2} turns out to be strongly indefinite and, as a consequence, classical variational results such as the Mountain Pass Theorem do not apply. However, it is still possible to get existence of solutions by exploiting the Linking Theorem of Benci and Rabinowitz \cite{BenciRabinowitz79} and reduction methods, see \cite{deFigueiredoFelmer94_2, HulshofVanderVorst93} as well as the surveys on this topic \cite{BonheureDosSantosTavares14, Ruf08} and the references therein.
More precisely, the following hyperbola
\begin{equation}\label{CH} \frac{1}{p+1} + \frac{1}{q+1} = \frac{N-2}{N}, \end{equation}
which has been introduced by Mitidieri \cite{Mitidieri93}, plays the role of critical threshold for \eqref{LaneEmden}, namely we have existence of solutions below and non existence above \eqref{CH}.
Note that the energy functional \eqref{I2} is well defined on $W^{1, s} \times W^{1, t}$ with $\frac{1}{s} + \frac{1}{t} =1$ provided that $\frac{1}{p+1} \ge \frac{1}{s} - \frac{1}{N}$ and $\frac{1}{q+1} \ge \frac{1}{t} - \frac{1}{N}$, due to the Sobolev embeddings \cite{AdamsFournier03}. Therefore, 
\[ \frac{1}{p+1} + \frac{1}{q+1} \ge1- \frac{2}{N} = \frac{N-2}{N} \]
namely $(p, q)$ lies below \eqref{CH}.
The question of existence / non existence of solutions to \eqref{LaneEmden} is still not entirely answered if $\Omega = \mathbb{R}^N$, except for the radial case (see \cite{Mitidieri93, SerrinZou98, SerrinZou94}) and for dimensions $N=3$ \cite{PolacikQuittnerSouplet07}  and $N=4$ \cite{Souplet09}. 

\noindent A natural extension to \eqref{LaneEmden} is the following
\begin{equation}\label{LE}
\begin{cases}
\begin{aligned}
(-\Delta)^{\alpha} u=\abs{v}^{q-1} v \\
(-\Delta)^{\beta} v= \abs{u}^{p-1} u
\end{aligned} \text{ in } \Omega \subseteq \mathbb{R}^N \\
B(u, v)=0 \text{ on } \partial \Omega
\end{cases}
\end{equation}
where $\alpha, \beta \in \mathbb{N}$, $\alpha, \beta \ge 1$ and $B(u,v)=0$ represents boundary conditions if any. The situation in which two different polyharmonic operators, namely $\alpha \ne \beta$, are taken into account seems to be the most challenging, since it does not exhibit a variational structure. \\
Consider first $\Omega=\mathbb{R}^N$. If $\alpha=\beta \ne 1$, it was proved in \cite{LiuGuoZhang06_2} that there exist infinitely many radially symmetric classical solutions, provided $(p,q)$  lies above the corresponding critical hyperbola
\begin{equation}\label{CHalpha}  \frac{1}{p+1} + \frac{1}{q+1} = \frac{N-2\alpha}{N}, \end{equation}
whereas non existence of positive radial solutions is showed in  \cite{CaristiDAmbrosioMitidieri08} if $p, q>1$ and $(p,q)$ is below the critical hyperbola \eqref{CHalpha}. 
In the non radial case only partial results are known, see \cite{Zhang07, LiuGuoZhang06, ArthurYan16}, see also \cite{MitidieriPohozaev01, CaristiDAmbrosioMitidieri08}, where the authors prove non existence results for positive supersolutions to \eqref{LE}. \\
As for $\Omega$ bounded, in \cite{Sirakov08, Sirakov05} the problem of existence to \eqref{LE} with Navier boundary conditions, namely 
\begin{align*} \Delta^{r} u&=0, \, r=0, \dots, \alpha-1, \text{ on } \partial \Omega \\
\Delta^{r}v&=0, \, r=0, \dots, \beta-1, \text{ on } \partial \Omega ,\end{align*}
is considered (actually more general operators are taken into account). In \cite{LiuGuoZhang06_2} a non existence result is established above the hyperbola \eqref{CHalpha} for positive classical radial solutions to \eqref{LE} with Navier boundary conditions, if $\alpha=\beta$ and $\Omega$ is a star-shaped domain.
In \cite{HulshofVanderVorst93} the existence of solutions to \eqref{LE} with $\alpha=\beta$ in the subcritical case, namely below \eqref{CHalpha}, with $p, q>1$ and Navier boundary conditions, is stated.

\noindent At the best of our knowledge, the Dirichlet case \eqref{LE_D} has not been considered yet,  not even in the case $\alpha=\beta$. 
It is worth pointing out that Dirichlet boundary conditions are physically relevant (for example, they appear in the modeling of the clamped plate) as well as mathematically challenging: indeed, differently from the Navier case, the polyharmonic operator combined with Dirichlet boundary conditions is not equivalent to the iteration of the Laplace operator with Dirichlet boundary conditions; therefore, system \eqref{LE_D} cannot be split into a system of $\alpha+\beta$ equations. Furthermore, we will have to deal with the lack of a maximum principle in generic bounded domains. 

Let us state our main results: 

\begin{theorem}\label{teorema2}
Consider the following
\begin{equation}\label{LaneEmden_limitato}
\begin{cases}
\begin{aligned}
(-\Delta)^{\alpha} u=\abs{v}^q \\
(-\Delta)^{\beta} v= \abs{u}^p
\end{aligned} \text{ in } B_1 \subset \mathbb{R}^N \\
\frac{\partial^{r} u}{\partial \nu^{r}}=0, \, r=0, \dots, \alpha-1,  \text{ on } \partial B_1 \\
\frac{\partial^{r} v}{\partial \nu^{r}}=0, \, r=0, \dots, \beta-1, \text{ on } \partial B_1.
\end{cases}
\end{equation}
where $B_1$ is the unitary ball in $\mathbb{R}^N$.
If the only solution to \eqref{LaneEmden_limitato} is the trivial one, and if $pq>1$, then there exists a classical nontrivial radially symmetric solution to the system 
\begin{equation}\label{sysnonexistence_intro}
\begin{cases}
\begin{aligned}
(-\Delta)^{\alpha} u = \abs{v}^q \\
(-\Delta)^{\beta} v = \abs{u}^p 
\end{aligned} \text{ in $\mathbb{R}^N$.}
\end{cases}
\end{equation}
\end{theorem}

\noindent As a consequence, we have the following
\begin{corollary}\label{teorema}
Assume that $p,q >1$ and that one of the following conditions holds:
\begin{itemize}
\item[(i)]  $2\beta q + N + 2\alpha pq - N pq \ge 0 \text{ or } 2\alpha p + N + 2\beta pq - N pq \ge 0$
\item[(ii)] $p, q< \min \{ \frac{N+2\alpha}{N-2\beta}, \frac{N+2\beta}{N-2\alpha} \}$.
\end{itemize}
Then there exists a classical nontrivial solution to \eqref{LaneEmden_limitato}. 
\end{corollary}
\begin{remark}
Note that  if one takes $\alpha=\beta=1$ then \textit{(i)} becomes
\[  \frac{1}{p+1} + \frac{1}{q+1} \ge 1- \frac{2}{N-2} \max \left( \frac{1}{p+1}, \frac{1}{q+1} \right) \]
which is the so called Serrin curve (see \cite{Mitidieri93, MitidieriPohozaev01}) and constitutes the threshold between existence and non existence of supersolutions to \eqref{LaneEmden}. 
\end{remark}
\begin{corollary}\label{teoremabis}
Consider the case $\alpha=\beta$ in \eqref{LaneEmden_limitato}. Assume that $p, q \ge 1$ not both equal to $1$ and that $(p,q)$ lies below the critical hyperbola \eqref{CHalpha}. Then there exists a classical nontrivial solution to \eqref{LaneEmden_limitato}. 
\end{corollary}

\begin{figure}
\centering
\includegraphics[width=0.7\textwidth]{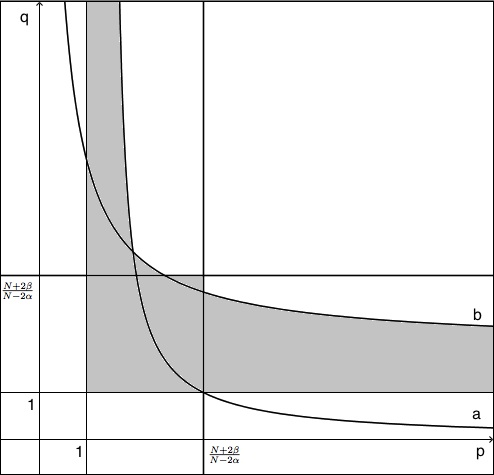}
\caption{The grey area is the region of existence of solutions to \eqref{LaneEmden_limitato} given by \autoref{teorema} in the case $\frac{N+2\alpha}{N-2\beta}> \frac{N+2\beta}{N-2\alpha}$. The curves $a$ and $b$ represent respectively $2\beta q + N + 2\alpha pq - N pq =0$ and $2\alpha p + N + 2\beta pq - N pq =0$. }
\label{curve1}
\end{figure}
\begin{figure}
\centering
\includegraphics[width=0.7\textwidth]{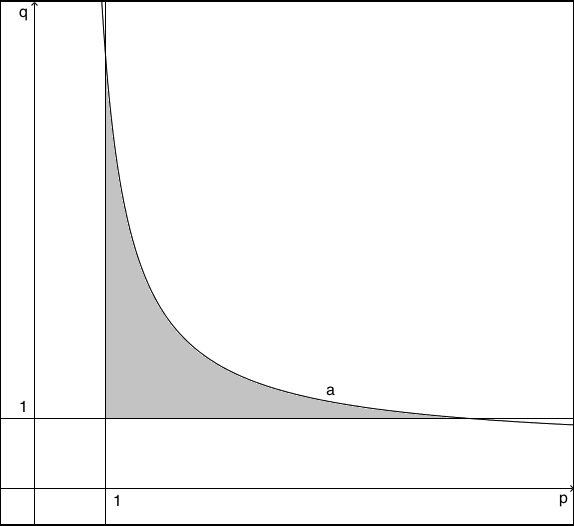}
\caption{The grey area is the region of existence of solutions to \eqref{LaneEmden_limitato} given by \autoref{teoremabis}. The curve $a$ represents the critical hyperbola \eqref{CHalpha}. }
\label{curve2}
\end{figure}

\noindent Uniqueness of solutions to \eqref{LaneEmden_limitato} with $\alpha, \beta \in \mathbb{N}$ turns out to be a challenging problem. So far, only the second order case $\alpha=\beta=1$ has been taken into account, see \cite{Dalmasso04}.
Actually, only partial results are known even in the case of polyharmonic equations. We refer to \cite{Dalmasso95} for the proof of uniqueness of positive solutions to 
\[ \begin{cases}
\Delta^2 u = \abs{u}^{p-1} u \,  &B_1\\
u=\frac{\partial u}{\partial \nu}=0 \, &\partial B_1
\end{cases} \]
see also \cite{FerreroGazzolaWeth07}, and \cite{Dalmasso99} for uniqueness of solutions to polyharmonic equations of higher order with sublinear nonlinearities. 

\noindent However, in the case $\alpha=2$ and $\beta=1$ we prove the following 
 
\begin{theorem}\label{coruniq}
Let us consider
\begin{equation}\label{21}
\begin{cases}
\begin{aligned}
\Delta^2 u=\abs{v}^q \\
-\Delta v= \abs{u}^p
\end{aligned} \text{ in } B_1 \\
u=\frac{\partial u}{\partial \nu}=v=0 \text{ on } \partial B_1
\end{cases}
\end{equation}
where $p, q > 1$.  
Then there exists at most one nontrivial solution to \eqref{21}.
\end{theorem}

\noindent In order to prove \autoref{teorema2}, we borrow a few ideas from \cite{AziziehClementMitidieri02}, where an Hamiltonian system with $(p,q)$-Laplace operators was considered. 
A continuation method together with a priori estimates are key ingredients. As for \autoref{coruniq}, we exploit an idea originally due to Gidas, Ni and Nirenberg \cite{GidasNiNirenberg79}, see also \cite{CuiWangShi07, Dalmasso04, Dalmasso95}. 
Since the argument is quite involved, let us split it into four steps:

\emph{Step 1.}
By exploiting the Leray-Schauder degree and a continuation argument, one proves that if the only solution to \eqref{LaneEmden_limitato} is the trivial one, then there exists an unbounded sequence of solutions to a system $S_t$ depending on a scaling parameter $t$. 

\emph{Step 2.}
One performs a blow-up analysis; more precisely, by assuming the existence of a sequence of functions as in Step 1 and suitable geometric requirements on the position of global maxima of these functions, one gets the existence of a classical radial nontrivial solution to system \eqref{sysnonexistence_intro}. 
 
\emph{Step 3.}
The moving planes procedure together with Gidas-Ni-Nirenberg arguments are used to derive information on the location of global maxima of solutions to $S_t$ in order to establish the geometric assumptions required in Step 2. Actually, solutions to $S_t$ turn out to be radially symmetric and strictly decreasing in the radial variable.

\emph{Step 4.} 
By gluing results in the previous steps, one gets \autoref{teorema2}.
We then prove \autoref{teorema} and \autoref{teoremabis} by contradiction: in view of \autoref{teorema2}, one can exploit non existence results for \eqref{sysnonexistence_intro} proved in \cite{MitidieriPohozaev01, LiuGuoZhang06}. Moreover, we give the proof of \autoref{coruniq} by showing that radially symmetric, strictly decreasing and positive solutions to \eqref{21} are unique.

This paper is organized as follows: in Section 2 we recall some useful properties of the polyharmonic operator and non existence results in \cite{MitidieriPohozaev01, LiuGuoZhang06}; Sections 3-6 are devoted to the proof of the main Theorems, each Section corresponding to a Step of the proof as above. Finally, in Appendix A we recall for completeness a detailed proof of the non existence result by Mitidieri and Pohozaev \cite{MitidieriPohozaev01} exploited in Step 4, adapted to the case under consideration.

\section{Preliminaries}\label{sec:preliminaries}
Let us first recall some properties of the polyharmonic operator \cite{GazzolaGrunauSweers10}.
\begin{theorem}\label{lemKalpha}
Let us consider the following
\begin{equation}\label{K}
\begin{cases}
(-\Delta)^{\alpha} u = f &\Omega \\
\frac{\partial^{r} u}{\partial \nu^{r}}=0, \, r=0, \dots, \alpha-1,  &\text{ on } \partial \Omega
\end{cases}
\end{equation}
where $\Omega$ is a smooth bounded domain of $\mathbb{R}^N$, $N > 2\alpha$. If $f \in C^{r}(\bar{\Omega})$ with $r\ge 1$ then there exists a solution to \eqref{K} which is in $C^{2\alpha, \gamma}(\bar{\Omega})$ for any $\gamma \in (0,1)$. Moreover, the map $K_{\alpha} \colon C^{r}(\bar{\Omega}) \to C^{2\alpha}(\bar{\Omega})$ defined as $K_{\alpha}(f)=u$ is compact. 
\end{theorem}
\begin{proof}
By elliptic regularity we know that if $f \in W^{k-2\alpha, p}(\Omega)$ for $k \ge 2\alpha$ then there exists a unique nontrivial solution to \eqref{K} such that
\begin{equation}\label{stimanorma} \norm{u}_{W^{k,p}} \le C \norm{f}_{W^{k-2\alpha, p}}. \end{equation}
Assume $f \in C^{r}(\bar{\Omega})$ with $r\ge 1$. Then $f \in W^{r, p}(\Omega)$ for any $p$. Hence, by \eqref{stimanorma} one has 
\[ \norm{u}_{W^{r+2\alpha,p}} \le C \norm{f}_{W^{r, p}} \]
for any $p$. 
By taking $p$ large enough, one has $W^{r+2\alpha,p}(\Omega) \hookrightarrow C^{r+2\alpha-1, \gamma}(\bar{\Omega})$ compactly for any $\gamma \in (0,1)$ (see \cite{AdamsFournier03}) and since $r\ge 1$ then $u \in C^{2\alpha, \gamma}(\bar{\Omega}) \subset C^{2\alpha}(\bar{\Omega})$.
\end{proof}
\noindent The next Theorem establishes a positivity preserving property due to Boggio \cite{Boggio1901}. 
\begin{theorem}[see Proposition 3.6 in \cite{GazzolaGrunauSweers10}]\label{boggio}
Let us assume $\Omega= B_1$ or $\Omega=\mathbb{R}^N$ in \eqref{K}. If $f \ge 0$, then $K_{\alpha}(f) \ge 0$ as well. In particular, if $f \ge g$, then $K_{\alpha}(f) \ge K_{\alpha}(g)$. 
\end{theorem}
\begin{remark}
\autoref{boggio} will be crucial in the sequel and that is the main reason why we restrict ourself to the case $\Omega=B_1$. Moreover, due to \autoref{boggio} one has that any nontrivial solution to \eqref{LaneEmden_limitato} is positive on $B_1$.
\end{remark}
\begin{theorem}[Lemma 7.9 in \cite{GazzolaGrunauSweers10}]\label{hopf}
If $f \in C^0$, $f \ge 0$, $x_0\in \partial B_1$ and $\mu \in \mathbb{R}^N$ such that $\mu \cdot x_0 <0$, then $\frac{\partial^{\alpha}u}{\partial \mu^{\alpha}} (x_0) >0$, where $u=K_{\alpha} (f)$.
\end{theorem}

\noindent In what follows, we recall some non existence results to  
\begin{equation}\label{LaneEmden_mn}
\begin{cases}
\begin{aligned}
(-\Delta)^{\alpha} u = \abs{v}^{q} \\
(-\Delta)^{\beta} v = \abs{u}^{p}
\end{aligned} & \mathbb{R^N}.
\end{cases}
\end{equation}
 
\noindent The next Theorem is established in \cite{MitidieriPohozaev01}; since the proof is given for general higher order operators, we give for completeness a detailed proof in Appendix A in the case of the polyharmonic operators. 
\begin{theorem}\label{GZL}
Let us consider the following
\begin{equation}\label{LaneEmden_mn}
\begin{cases}
\begin{aligned}
(-\Delta)^{\alpha} u \ge \abs{v}^{q} \\
(-\Delta)^{\beta} v \ge \abs{u}^{p}
\end{aligned} & \mathbb{R^N}
\end{cases}
\end{equation}
with $N> 2 \alpha, 2 \beta$.
If $p,q >1$ and $2\beta q + N + 2\alpha pq - N pq \ge 0$ or $2\alpha p + N + 2\beta pq - N pq \ge 0$, then there exist no weak solutions to \eqref{LaneEmden_mn}.
\end{theorem}
\begin{remark}
In \cite{LiuGuoZhang06} the same result is proved for classical solutions, whereas in \cite{CaristiDAmbrosioMitidieri08} an alternative proof is given by exploiting suitable representation formulas for \eqref{LaneEmden_mn}. 
\end{remark}
\noindent Let us now state without proof the next results. 
\begin{theorem}[see Theorem 1.2' in \cite{LiuGuoZhang06}]\label{GLZ2}
If $1 < p, q< \min \{ \frac{N+2\alpha}{N-2\beta}, \frac{N+2\beta}{N-2\alpha} \}$ then the only classical solution to 
\begin{equation*}
\begin{cases}
\begin{aligned}
(-\Delta)^{\alpha} u = \abs{v}^{q} \\
(-\Delta)^{\beta} v = \abs{u}^{p}
\end{aligned} & \mathbb{R^N}
\end{cases}
\end{equation*}
is the trivial one.
\end{theorem}
\begin{theorem}[see Theorem 1.1 in \cite{LiuGuoZhang06} and Theorem 5.1 in \cite{CaristiDAmbrosioMitidieri08}] \label{nethradial}
Let us assume that 
\[ \frac{1}{p+1} + \frac{1}{q+1} > \frac{N-2\alpha}{N} \]
and $p, q \ge 1$ and not both equal to $1$. Then the only radially symmetric, classical solution to 
\begin{equation*}
\begin{cases}
\begin{aligned}
(-\Delta)^{\alpha} u = \abs{v}^q \\
(-\Delta)^{\alpha} v = \abs{u}^p 
\end{aligned} \text{ in $\mathbb{R}^N$} 
\end{cases}
\end{equation*}
is the trivial one. 
\end{theorem}

\section{A continuum of solutions}
This Section is devoted to prove 
\begin{proposition}\label{sequenzenonlimitate}
Let $p, q$ such that $pq>1$. 
Denote with $\mathcal{C}$ the component in $\mathbb{R}^+ \times C_0^{2\alpha}(\bar{B}_1) \times C_0^{2\beta}(\bar{B}_1)$ of solutions $(t, u, v)$ to 
\begin{equation}\label{eqnt}
\begin{cases}
\begin{aligned}
(-\Delta)^{\alpha} u &= (t+ \abs{v})^{q} \\
(-\Delta)^{\beta} v &= (t^{\vartheta}+\abs{u})^{p}
\end{aligned} \,  \text{in } B_1\\
\frac{\partial^{r} u}{\partial \nu^{r}}=0, \, r=0, \dots, \alpha-1,  \text{ on } \partial B_1 \\
\frac{\partial^{r} v}{\partial \nu^{r}}=0, \, r=0, \dots, \beta-1, \text{ on } \partial B_1
\end{cases}
\end{equation}
containing $(0,0,0)$, where $\vartheta \in (1/p, q)$. If 
\[ \mathcal{C} \cap (\{0\}\times C_0^{2\alpha}(\bar{B}_1) \times C_0^{2\beta}(\bar{B}_1))=\{(0,0,0)\} \]
then $\mathcal{C}$ is unbounded in $\mathbb{R}^+ \times C_0^{2\alpha}(\bar{B}_1) \times C_0^{2\beta}(\bar{B}_1)$. 
\end{proposition}
\noindent The proof of \autoref{sequenzenonlimitate} needs a few preliminary results.
\begin{lemma}[see Lemma A.2 in {\cite{AziziehClement02}}]\label{lemgrado}
Let $(E, \norm{\cdot})$ be a real Banach space. Let $G \colon \mathbb{R}^+ \times E \to E$ be continuous and compact. Suppose, moreover, $G$ satisfies 
\begin{itemize}
\item[(a)] $G(0,0)=0$
\item[(b)] there exists $R>0$ such that 
\begin{itemize}
\item[(i)] $u\in E$, $\norm{u} \le R$ and $u=G(0,u)$ implies $u=0$
\item[(ii)] $deg(Id-G(0, \cdot), B_R, 0)=1$.
\end{itemize}
\end{itemize}
Let $J$ denote the set of solutions to the problem $u=G(t,u)$ in $\mathbb{R}^+ \times E$. Let $\mathcal{C}$ denote the component of $J$ containing $(0,0)$. If 
\[ \mathcal{C} \cap (\{0\} \times E) = \{(0, 0)\} \]
then $\mathcal{C}$ is unbounded in $\mathbb{R}^+ \times E$. 
\end{lemma}

\begin{lemma}\label{lemstimedalbasso}
Let $p, q$ such that $pq>1$. Then there exists a real number $R>0$ such that if $(\lambda, u ,v ) \in [0,1] \times C_0^{2\alpha}(\bar{B}_1) \times C_0^{2\beta}(\bar{B}_1)$ is a solution to
\begin{equation}\label{eqnlemma}
\begin{cases}
u = K_{\alpha} (\lambda \abs{v}^{q}) \\
v = K_{\beta} (\lambda \abs{u}^{p}) \\
u \ne 0 \text{ or } v\ne 0
\end{cases} 
\end{equation}
then $\norm{u}_{\infty} >R$ and $\norm{v}_{\infty} >R$.
\end{lemma}

\begin{proof}
By  \autoref{boggio} one has $K_{\alpha} \left(  \left ( \frac{\abs{v}}{\norm{v}_{\infty}} \right)^q \right) \le K_{\alpha} (1)$, hence
\[
\abs{u} = \abs{K_{\alpha} (\lambda \abs{v}^{q})} \le \abs{K_{\alpha} \left( \norm{v}^q_{\infty} \left ( \frac{\abs{v}}{\norm{v}_{\infty}} \right)^q \right)} 
\le \norm{v}^q_{\infty} \abs{K_{\alpha} (1)} \le C_1 \norm{v}^q_{\infty}
\]
thus
\[ \norm{u}_{\infty} \le C_1 \norm{v}^q_{\infty} \]
and similarly
\[ \norm{v}_{\infty} \le C_2 \norm{u}^p_{\infty}. \]
Then $\norm{u}_{\infty} \le C_1 C_2^q \norm{u}_{\infty}^{pq}$ and therefore $\norm{u}_{\infty} \ge R$; similarly for $v$.
\end{proof}

\begin{proof}[Proof of \autoref{sequenzenonlimitate}]
We apply \autoref{lemgrado}: let us define
\[ G \colon [0,+\infty)  \times C_0^{2\alpha}(\bar{B}_1) \times C_0^{2\beta}(\bar{B}_1) \to C_0^{2\alpha}(\bar{B}_1) \times C_0^{2\beta}(\bar{B}_1) \]
as follows
\[ G(t, u, v)=(K_{\alpha}(t+\abs{v})^q, K_{\beta}(t^{\vartheta}+\abs{u})^p). \]
The operator $G$ is continuous and compact since $K_{\alpha}, K_{\beta}$ have these properties by \autoref{lemKalpha}. Note that $v \in C^{2\beta}$ implies $(v+t)^q \in C^{2\beta}$ and thus by \autoref{lemKalpha} with $r=2\beta>1$ one has $K_{\alpha}(v+t)^q \in C^{2\alpha}$. 
Moreover $G(0,0,0)=(0,0)$.

\noindent Hypothesis $(b)(i)$ of \autoref{lemgrado} is satisfied by \autoref{lemstimedalbasso} with $\lambda=1$.

\noindent Let $B_{\alpha}(0,R)$ be the ball in $C_0^{2\alpha}(\bar B_1)$ of radius $R$ and let us define an homotopy $h \colon [0,1] \times \overline{B_{\alpha}(0,R)} \times \overline{B_{\beta}(0,R)} \to C_0^{2\alpha}(\bar{B}_1) \times C_0^{2\beta}(\bar{B}_1)$ by 
\[h(\lambda, u , v) \mapsto (K_{\alpha}(\lambda\abs{v}^q), K_{\beta}(\lambda \abs{u}^p)). \]
By \autoref{lemKalpha} the map $h$ is continuous and compact, $h(1, \cdot, \cdot)=G(0, \cdot, \cdot)$, $h(0, \cdot, \cdot)=(0,0)$ and by \autoref{lemstimedalbasso} one has $h(\lambda, u, v) \ne (u,v)$ for all $(u,v) \in \partial(B_{\alpha}(0,R) \times B_{\beta}(0,R))$. Therefore hypothesis $(b)(ii)$ of \autoref{lemgrado} is also satisfied. Indeed, one uses the homotopy invariance of the Leray-Schauder degree (see e.g.~Theorem 2.1 (iii) in \cite{Mawhin99}): 
\begin{multline*} deg(Id-G(0, \cdot, \cdot), B_R, 0)=deg(Id - h(1, \cdot, \cdot),  B_R, 0)\\= deg(Id - h(0, \cdot, \cdot),  B_R, 0)=deg(Id,  B_R, 0)= 1.  \qedhere \end{multline*}
\end{proof}

\section{Blow-up analysis}
The main result of this Section is the following 
\begin{proposition}\label{lemblowup}
Let $(t_n, u_n, v_n)$ be a sequence of solutions to \eqref{eqnt} in $\mathbb{R}^+ \times C_0^{2\alpha}(\bar{B}_1) \times C_0^{2\beta}(\bar{B}_1)$ with $pq >1$ and $\vartheta \in (1/p, q)$ fixed such that 
\begin{equation}\label{successione1} t_n + \norm{u_n}_{\infty} + \norm{v_n}_{\infty} \to \infty. \end{equation}
Suppose that there exist $\rho >0$ and $\{ x_n \}, \{ x_n' \} \in B_1$ satisfying $u_n(x_n)=\norm{u_n}_{\infty}$, $v_n(x_n')=\norm{v_n}_{\infty}$ and such that
\[ dist(x_n, \partial B_1) \ge \rho, \, dist(x_n', \partial B_1) \ge \rho . \]
Then there exists $(u,v) \in C^{2\alpha}(\mathbb{R}^N) \times C^{2\beta}(\mathbb{R}^N)$ nontrivial solution to
\begin{equation}\label{sysnonexistence}
\begin{cases}
\begin{aligned}
(-\Delta)^{\alpha} u = \abs{v}^q \\
(-\Delta)^{\beta} v = \abs{u}^p 
\end{aligned} \text{ in $\mathbb{R}^N$.} 
\end{cases}
\end{equation}
Moreover, if $u_n, v_n$ are radially symmetric and $x_n=x_n'=0$ for any $n$, then there exists a nontrivial radial solution to \eqref{sysnonexistence}.
\end{proposition}
\begin{proof}
Assume without loss of generality $\alpha, \beta$ even.
Let us prove first that there exists a subsequence such that 
\begin{equation}\label{successione}
\frac{t_n^{\vartheta}}{\norm{u_n}_{\infty}} \to 0 \text{ and } \frac{t_n}{\norm{v_n}_{\infty}} \to 0 
\end{equation}
with $\norm{u_n}_{\infty} > 0 $ and $\norm{v_n}_{\infty} >0$ for all $n$. Observe that if $u_n=0$, then $v_n=0$ and $t_n=0$, therefore by \eqref{successione1} we have that $u_n=0$ only for a finite number of indices $n$ and similarly for $v_n$, namely there exists a subsequence such that $u_n \ne 0, v_n \ne 0$ for any $n$. Two cases may occur: 
\begin{itemize}
\item if $t_n$ is bounded, then for example $\norm{u_n}_{\infty} \to \infty$, thus by \autoref{lemKalpha} $\norm{v_n}_{\infty} \to \infty$ as well and \eqref{successione} follows. 
\item $t_n \to \infty$; assume without loss of generality that $t_n>0$. Let us introduce the following change of variable:
\begin{align*} \tilde{u}_n&=\frac{u_n}{t_n^{\vartheta}}, \quad \lambda_n=t_n^{q-\vartheta} \\
 \tilde{v}_n&=\frac{v_n}{t_n}, \quad \mu_n=t_n^{\vartheta p - 1}. \end{align*}
Then for all $n$ one has 
\[
\begin{cases}
\begin{aligned}
(-\Delta)^{\alpha} \tilde{u}_n=\lambda_n(1+\abs{\tilde{v}_n})^q \ge \lambda_n \\
(-\Delta)^{\beta} \tilde{v}_n=\mu_n(1+ \abs{\tilde{u}_n})^p \ge \mu_n
\end{aligned} \text{ on } B_1 \\
\frac{\partial^{r} \tilde{u}_n}{\partial \nu^{r}}=0, \, r=0, \dots, \alpha-1,  \text{ on } \partial B_1 \\
\frac{\partial^{r} \tilde{v}_n}{\partial \nu^{r}}=0, \, r=0, \dots, \beta-1, \text{ on } \partial B_1
\end{cases}
\]
Moreover, since $\vartheta \in (1/p, q)$, then $\lambda_n, \mu_n \to \infty$. For any fixed $n$, let us denote by $(w_n, z_n) $ the solution to 
\[
\begin{cases}
\begin{aligned}
(-\Delta)^{\alpha}w_n = \lambda_n \\
(-\Delta)^{\beta}z_n=\mu_n
\end{aligned} \text{ on } B_1 \\
\frac{\partial^{r} w_n}{\partial \nu^{r}}=0, \, r=0, \dots, \alpha-1,  \text{ on } \partial B_1 \\
\frac{\partial^{r} z_n}{\partial \nu^{r}}=0, \, r=0, \dots, \beta-1, \text{ on } \partial B_1
\end{cases}
\]
Then one has $\tilde{u}_n \ge w_n$ and $\tilde{v}_n \ge z_n$ by the comparison principle (see \autoref{boggio}). Moreover, we claim
\[ \sup_n \norm{w_n}_{\infty}=\sup_n \norm{z_n}_{\infty} = + \infty. \]
Indeed, let us suppose by contradiction that $\sup_n \norm{w_n}_{\infty} \le c$. Then one has
\begin{equation}\label{eqnblowup} \norm{w_n}_{\alpha}^2 = \int_{B_1} \abs{\Delta^{\alpha/2}w_n}^2 = \lambda_n \int_{B_1} w_n \le  c \lambda_n. \end{equation}
However, for any $\varphi \ge 0, \ne 0$ one has 
\[ 0 < D=\int_{B_1} \varphi = \frac{1}{\lambda_n} \int_{B_1} \Delta^{\alpha/2} w_n \Delta^{\alpha/2} \varphi \]
and by \eqref{eqnblowup},
\[ 0< D \le \frac{1}{\lambda_n} \norm{w_n}_{\alpha} \norm{\varphi}_{\alpha} \le c^{\frac{1}{2}} \lambda_n^{-\frac{1}{2}} \norm{\varphi}_{\alpha} \]
which tends to $0$ as $n \to \infty$, a contradiction. Similarly for $z_n$. Then the claim holds and as a consequence $\sup_n \norm{\tilde u_n}_{\infty}=\sup_n \norm{\tilde v_n}_{\infty}= \infty$, which is equivalent to \eqref{successione}.
\end{itemize}
Now, let us consider $(t_n, u_n, v_n)$ which satisfies \eqref{successione}.
Let $A_n, B_n, C_n >0$ to be chosen in the sequel and assume first that $\norm{u_n}_{\infty}^{1/\tau} \ge \norm{v_n}_{\infty}^{1/\sigma}$ for any $n$, where $\tau=\frac{2\beta q + 2\alpha}{pq-1}$ and $\sigma=\frac{2 \alpha p + 2\beta}{pq-1}$. Define the following scaling
\[ \hat{u}_n(y)=\frac{u_n(C_n^{-1} y + x_n)}{A_n} \]
and
\[ \hat{v}_n(y)=\frac{v_n(C_n^{-1} y + x_n)}{B_n} \]
for any $y \in C_n(B_1 - x_n)=B(C_nx_n, C_n)$. Let  $\hat{\psi}_n(x)=\psi(C_n(x-x_n))$ and $\psi \in C_0^{\infty}$. Then 
\begin{multline*}
\int_{B(C_nx_n, C_n)} \Delta^{\alpha/2} \hat{u}_n \Delta^{\alpha/2} \psi \, dy =\int_{B_1} A_n^{-1} C_n^{N-2\alpha} \Delta^{\alpha/2} u_n \Delta^{\alpha/2} \hat{\psi}_n \, dx\\
=\int_{B_1} A_n^{-1} C_n^{N-2\alpha}(t_n+ \abs{v_n})^q\hat{\psi}_n \, dx \\
=\int_{B(C_nx_n, C_n)} A_n^{-1}B_n^q C_n^{-2\alpha}\left (\frac{t_n}{B_n} + \abs{\hat{v}_n}\right )^q \psi \, dy
\end{multline*}
and 
\[ \int_{B(C_nx_n, C_n)} \Delta^{\beta/2} \hat{v}_n \Delta^{\beta/2} \psi \, dy = \int_{B(C_nx_n, C_n)} B_n^{-1}A_n^p C_n^{-2\beta}\left (\frac{t_n^{\vartheta}}{A_n} + \abs{\hat{u}_n}\right )^p \psi \, dy.  \]
Now choose $A_n=C_n^{\tau}$, $B_n=C_n^{\sigma}$ and $C_n=\norm{u_n}_{\infty}^{1/\tau} + \norm{v_n}_{\infty}^{1/\sigma}$, so that
\begin{equation}\label{eqn_n} \begin{split}
\int_{B(C_nx_n, C_n)} \Delta^{\alpha/2} \hat{u}_n \Delta^{\alpha/2} \psi \, dy &=\int_{B(C_nx_n, C_n)} \left (\frac{t_n}{B_n} + \abs{\hat{v}_n}\right )^q \psi \, dy \\
\int_{B(C_nx_n, C_n)} \Delta^{\beta/2} \hat{v}_n \Delta^{\beta/2} \psi \, dy &= \int_{B(C_nx_n, C_n)} \left (\frac{t_n^{\vartheta}}{A_n} + \abs{\hat{u}_n}\right )^p \psi \, dy.
\end{split} \end{equation}
Hence, by \eqref{successione1} and \eqref{successione} one has $C_n \to \infty$, thus,
\begin{equation}\label{dist} dist(0, \partial B(C_nx_n, C_n)) = C_n dist(x_n, \partial B_1) \ge C_n \rho \to \infty. \end{equation}
Moreover, 
\[ 0 \le \frac{t_n}{B_n} = \frac{t_n}{{(\norm{u_n}_{\infty}^{1/\tau} + \norm{v_n}_{\infty}^{1/\sigma})}^{\sigma}} \le \frac{t_n}{\norm{v_n}_{\infty}} \to 0 \]
and 
\[ 0 \le \frac{t_n^{\vartheta}}{A_n} = \frac{t_n^{\vartheta}}{{(\norm{v_n}_{\infty}^{1/\sigma} + \norm{u_n}_{\infty}^{1/\tau})}^{\tau}} \le \frac{t_n^{\vartheta}}{ \norm{u_n}_{\infty}} \to 0 .\]
Let $B$ any closed ball. Then by \eqref{dist} $B$ is contained in $B(C_nx_n, C_n)$ for $n$ large enough. 
Moreover, since the embedding $C^{2\alpha, \gamma}(\bar{\Omega}) \hookrightarrow C^{2\alpha}(\bar{\Omega})$ is compact, see \cite{AdamsFournier03},  then $(\hat{u}_n, \hat{v}_n) \in C^{2\alpha, \gamma}(\bar B) \times C^{2\beta, \gamma}(\bar B)$ converges up to a subsequence to $(\hat u,\hat v)$ in $C^{2\alpha}(\bar B) \times C^{2\beta}(\bar B)$; by considering integrals in \eqref{eqn_n} on the ball $B$ and letting $n \to \infty$, one has 
\[ \begin{split}
\int_{B} \Delta^{\alpha/2}\hat u \Delta^{\alpha/2} \psi \, dy &=\int_{B} \abs{\hat v}^q \psi \, dy \\
\int_{B} \Delta^{\beta/2} \hat v \Delta^{\beta/2} \psi \, dy &= \int_{B} \abs{\hat u}^p \psi \, dy
\end{split} \]
for any $\psi \in C_0^{\infty}(B)$. 
Note that $(\hat u,\hat v) \ne (0,0)$: indeed, 
\begin{align*} \hat{u}_n^{1/\tau}(0)&=\frac{u_n(x_n)^{1/\tau}}{\norm{u_n}_{\infty}^{1/\tau} + \norm{v_n}_{\infty}^{1/\sigma}}\\&=\frac{\norm{u_n}_{\infty}^{1/\tau}}{\norm{u_n}_{\infty}^{1/\tau} + \norm{v_n}_{\infty}^{1/\sigma}} = \frac{1}{1+ \norm{v_n}_{\infty}^{1/\sigma}/\norm{u_n}_{\infty}^{1/\tau}} \ge \frac{1}{2} \end{align*}
and therefore $\hat u(0) \ne 0$. 
Let us now take a larger ball $\tilde B$ and repeat the argument on the subsequence obtained at the previous step. 
Taking balls larger and larger and iterating the reasoning, we get two Cantor diagonal subsequences converging on all compacts of $\mathbb{R}^N$ to nontrivial functions $(\hat u,\hat v) \in C^{2\alpha}(\mathbb{R}^N) \times C^{2\beta}(\mathbb{R}^N)$ satisfying
\[ \begin{split}
\int_{\mathbb{R}^N} \Delta^{\alpha/2}\hat u \Delta^{\alpha/2} \psi \, dy &=\int_{\mathbb{R}^N} \abs{\hat v}^q \psi \, dy \\
\int_{\mathbb{R}^N} \Delta^{\beta/2}\hat v \Delta^{\beta/2} \psi \, dy &= \int_{\mathbb{R}^N} \abs{\hat u}^p \psi \, dy.
\end{split} \]
If $\norm{u_n}_{\infty}^{1/\tau} \le \norm{v_n}_{\infty}^{1/\sigma}$ we take $x_n'$ instead of $x_n$ in the definition of $\hat{u}_n$ and $\hat{v}_n$ and at the end we observe 
\[ \hat{v}_n^{1/\sigma}(0) \ge \frac{1}{2} .\]
This concludes the proof.
\end{proof}

\section{A priori estimates}
Let us consider the problem
\begin{equation}\label{problemfg}
\begin{cases}
\begin{aligned}
(-\Delta)^{\alpha} u=g(v) \\
(-\Delta)^{\beta} v= f(u)
\end{aligned} \text{ in } B_1 \\
\frac{\partial^{r} u}{\partial \nu^{r}}=0, \, r=0, \dots, \alpha-1,  \text{ on } \partial B_1 \\
\frac{\partial^{r} v}{\partial \nu^{r}}=0, \, r=0, \dots, \beta-1, \text{ on } \partial B_1
\end{cases}
\end{equation}
where $f, g \colon [0, \infty) \to \mathbb{R}$ are continuous, positive and non decreasing. 
The aim of this Section is to obtain information on the position of global maxima of solutions to \eqref{problemfg}, in order to apply \autoref{lemblowup}. More precisely, we show
\begin{proposition}\label{propestimates}
Let $(u, v) \in C_0^{2\alpha}(\bar B_1) \times C_0^{2\beta}(\bar B_1)$ nontrivial solution to \eqref{problemfg}. Then it is radially symmetric and strictly decreasing in the radial variable. In particular, $u$ and $v$ attain their maximum at $0$.
\end{proposition}

In order to prove \autoref{propestimates}, we apply the moving planes technique \cite{BerestyckiNirenberg91} and we adapt the classical symmetry result by Gidas, Ni, Nirenberg \cite{GidasNiNirenberg79}, by extending to the case of systems a few proofs of Section 7 in \cite{GazzolaGrunauSweers10}, where the case of a single equation is considered, see also \cite{BerchioGazzolaWeth08}. 
Define 
\begin{align*} T_{i, \lambda}&=\{ x=(x_1, \dots, x_N) \in \mathbb{R}^N: x_i=\lambda \} \\
\Sigma_{i, \lambda}&=\{ x=(x_1, \dots, x_N)\in B_1: x_i<\lambda \} \end{align*}
where $\lambda \in [0, 1]$, and let $x^{i, \lambda}$ denote the reflection of $x$ about $T_{i, \lambda}$. 

The next Lemmas constitute the preparation to the moving planes procedure.
\begin{lemma}[see Lemma 7.5 in \cite{GazzolaGrunauSweers10}]\label{GGS1}
Assume $h \in L^{\infty}(B_1)$ and let $u \in H_0^{\alpha}(B_1)$ satisfy
\[ \braket{u, v}_{H_0^{\alpha}} = \int_{B_1} h v \, dx \quad \text{ for all } v \in H_0^{\alpha}(B_1) \]
i.e., $u$ is a weak solution to $(-\Delta)^{\alpha} u =h $ in $B_1$ under Dirichlet boundary conditions. 
Then $u$ satisfies
 \[ D^{\gamma} u(x)=\int_{B_1} D_x^{\gamma} G^{\alpha}(x, y) h(y) \, dy \quad \text{for every } x \in \bar{B}_1 \]
 where $G^{\alpha}(x, y)$ is the Green function of $(-\Delta)^{\alpha}$ on $B_1$ with Dirichlet boundary conditions.
\end{lemma}
\begin{remark}
Note that such a function $u$ exists due to \eqref{stimanorma}, see Theorem 2.20 in \cite{GazzolaGrunauSweers10}. 
\end{remark}
\begin{lemma}[Lemma 7.7 and 7.8 in \cite{GazzolaGrunauSweers10}]\label{GGS1bis}
For all $x, y \in \Sigma_{i, \lambda}$  $x \ne y$, we have 
\[ G^{\alpha}(x,y) > \max \{ G^{\alpha}(x, y^{i, \lambda}), G^{\alpha}(x^{i, \lambda}, y) \} \]
\[ G^{\alpha}(x, y)- G^{\alpha}(x^{i, \lambda} y^{i, \lambda}) > \abs{ G^{\alpha}(x, y^{i, \lambda}) - G^{\alpha}(x^{i, \lambda}, y)}. \]
Moreover, for every $x \in B_1 \cap T_{i, \lambda}$ and $y \in \Sigma_{i, \lambda}$ we have
\begin{equation}\label{eqnG} \partial_{x_i}G^{\alpha}(x,y)<0 \quad \text{ and} \quad  \partial_{x_i}G^{\alpha}(x,y) +  \partial_{x_i}G^{\alpha}(x, y^{i, \lambda}) \le 0. \end{equation}
The second inequality in \eqref{eqnG} is strict if $\lambda >0$.
\end{lemma}
\noindent Let us define 
\[ \tilde{f}(s)=\begin{cases}
f(s) & \text{if } s>0 \\
0 & \text{if } s=0
\end{cases} 
\quad
\tilde{g}(s)=\begin{cases}
g(s) & \text{if } s>0 \\
0 & \text{if } s=0.
\end{cases} 
\]
Note that $f(0) \ne 0$, $g(0)\ne 0$ in general: indeed, we will apply the results of this Section to $f(u)=(t+\abs{u})^p$ and $g(v)=(t+\abs{v})^q$. 
From now on, let us extend $u, v$ out of $B_1$ by imposing $u=v=0$.
\begin{lemma}\label{GGS2}
Let $(u, v) \in C_0^{2\alpha}(\bar B_1) \times C_0^{2\beta}(\bar B_1)$, nontrivial solution to \eqref{problemfg}. Suppose $u(y) \ge u(y^{i, \lambda})$ and $v(y) \ge v(y^{i, \lambda})$ for all $y \in \Sigma_{i, \lambda}$. Then the following inequalities hold:
\begin{itemize}
\item $f(u(y)) \ge \tilde{f} (u(y^{i, \lambda})) \ge 0$ and $g(v(y)) \ge \tilde{g} (v(y^{i, \lambda})) \ge 0$ for all $y \in \Sigma_{i, \lambda}$
\item there exist two nonempty open sets $\mathcal{O}^u_{i, \lambda}, \mathcal{O}^v_{i, \lambda} \subset \Sigma_{i, \lambda}$ such that $f(u(y)) > \tilde{f}(u(y^{i, \lambda}))$ or $\tilde{f}(u(y^{i, \lambda})) >0$ for all $y \in \mathcal{O}^u_{i, \lambda}$ and  $g(v(y)) > \tilde{g}(v(y^{i, \lambda}))$ or $\tilde{g}(v(y^{i, \lambda})) >0$ for all $y \in \mathcal{O}^v_{i, \lambda}$.
\end{itemize}
\end{lemma}
\begin{proof}
First note that $u, v>0$ in $B_1$ due to \autoref{boggio}. 
The inequalities $f(u(y)) \ge \tilde{f} (u(y^{i, \lambda})) \ge 0$ and $g(v(y)) \ge \tilde{g} (v(y^{i, \lambda})) \ge 0$ follow from the monotonicity and positivity assumptions on $f, g$. \\
For the second statement it is enough to show that $f(u) \ne 0$ and $g(v)\ne 0$ in $ \Sigma_{i, \lambda}$. By contradiction, if $f(u)=0$ on $\Sigma_{i, \lambda}$ then the above inequalities  imply $\tilde{f}(u(y^{i,\lambda}))=0$, however since $u >0$ this means $f(u)=0$ on $B_1$. In turn, this implies $(-\Delta)^{\beta} v=0$, thus $v=0$, which contradicts the positivity of $v$. Similarly for $g(v)$.
\end{proof}
The following result will allow us to slide the hyperplane. 
\begin{lemma}\label{GGS3}
Let $(u, v)\in C_0^{2\alpha}(\bar B_1) \times C_0^{2\beta}(\bar B_1)$ nontrivial solution to \eqref{problemfg}. Suppose $u(x) \ge u(x^{i, \lambda})$ and $v(x) \ge v(x^{i, \lambda})$ for all $x \in \Sigma_{i, \lambda}$. Then there exists $\gamma \in (0, \lambda)$ such that $\frac{\partial u}{\partial x_i} <0, \frac{\partial v}{\partial x_i} <0$ on $T_{i, l} \cap B_1$ for all $l \in (\lambda - \gamma, \lambda)$. 
\end{lemma}
\begin{proof}
For all $x \in T_{i, \lambda} \cap B_1$, by \autoref{GGS1}, 
\begin{multline*} \frac{\partial u}{\partial x_i}(x)= \int_{B_1} \partial_{x_i} G^{\alpha} (x, y) g(v(y)) \, dy\\
=\int_{\Sigma_{i, \lambda}} [\partial_{x_i} G^{\alpha}(x,y)g(v(y)) + \partial_{x_i} G^{\alpha} (x, y^{i, \lambda}) \tilde{g}(v(y^{i, \lambda}))] \, dy \end{multline*}
Note that $y^{i, \lambda}$ may be outside $B_1$, however $g(0) \ne 0$ in general, hence one has to consider $\tilde g$ in place of $g$ in the last integral above.
By \autoref{GGS2} two cases may occur: $g(v(y)) > \tilde{g} (v(y^{i, \lambda}))$ for all $y \in \mathcal{O}^v_{i, \lambda}$ or $\tilde{g}(v(y^{i, \lambda}))>0$ for all $y \in \mathcal{O}^v_{i, \lambda}$ . In the first case,
\[ \frac{\partial u}{\partial x_i}(x) < \int_{\Sigma_{i, \lambda}} (\partial_{x_i} G^{\alpha} (x, y) + \partial_{x_i} G^{\alpha} (x, y^{i, \lambda}) )\tilde{g}(v(y^{i, \lambda})) \, dy \le 0 \quad \text{for all } x \in T_{i, \lambda} \cap B_1. \]
In the second case, 
\[ \frac{\partial u}{\partial x_i}(x) \le \int_{\Sigma_{i, \lambda}} (\partial_{x_i} G^{\alpha} (x, y) + \partial_{x_i} G^{\alpha} (x, y^{i, \lambda}) )\tilde{g}(v(y^{i, \lambda})) \, dy < 0 \quad \text{for all } x \in T_{i, \lambda} \cap B_1. \]
In any case, 
\begin{equation}\label{proof<} \frac{\partial u}{\partial x_i} (x) <0 \quad \text{ for all } x \in T_{i, \lambda} \cap B_1. \end{equation}
We now proceed exactly as in Lemma 7.11 in \cite{GazzolaGrunauSweers10}. 
For any $y \in \mathbb{R}^N$ and any $a>0$ let us consider the cube centered at $y$, namely
\[ \mathcal{U}_{a}(y) = \{ x \in \mathbb{R}^N: \max_{1 \le i \le N} \abs{x_i - y_i}<a \}. \]
Then by \autoref{hopf}, for any $x_0 \in T_{i, \lambda} \cap \partial B_1$ we have
\[ (-1)^{\alpha}\left( \frac{\partial}{\partial x_i} \right)^{\alpha - 1} \frac{\partial u}{\partial x_i}(x_0)=\left(-\frac{\partial }{\partial x_i} \right)^{\alpha} u(x_0) >0. \]
From boundary conditions we also know that $\left(\frac{\partial}{\partial x_i} \right)^k u(x_0)=0$ for all $k=0, \dots, \alpha - 1$, and hence there exists $a>0$ such that 
\[ \frac{\partial u}{\partial x_i}(x) <0 \quad \text{ for all } x \in \mathcal{U}_{a}(x_0) \cap B_1. \]
By compactness of $T_{i, \lambda} \cap \partial B_1$ there exists $\bar a >0$ such that 
\[ \frac{\partial u}{\partial x_i}(x) <0 \quad \text{ for all } x \in A=\bigcup_{x_0 \in T_{i, \lambda} \cap \partial B_1} (\mathcal{U}_{\bar a} (x_0) \cap B_1) \]
Let us set $K=(T_{i, \lambda} \cap B_1) \setminus A$ and for $d >0$ consider $K_d=K-de_i$. In view of \eqref{proof<} and by compactness of $K$, there exists $\delta>0$ such that 
\[ \frac{\partial u}{\partial x_i} <0 \text{ on } K_d \text{ for all } d \in [0, \delta]. \]
Let $\gamma_u=\min \{ \bar a, \delta \}$. Then, $\frac{\partial u}{\partial x_i} <0$ on $T_{i, l} \cap B_1$ for all $l \in (\lambda - \gamma_u, \lambda)$.\\
Analogously, one gets $\gamma_v$ such that $\frac{\partial v}{\partial x_i} <0$ on $T_{i, l} \cap B_1$ for all $l \in (\lambda - \gamma_v, \lambda)$. The conclusion follows by taking $\gamma=\min \{ \gamma_u, \gamma_ v \}$.
\end{proof}

\begin{figure}
\centering
\includegraphics[width=0.7\textwidth]{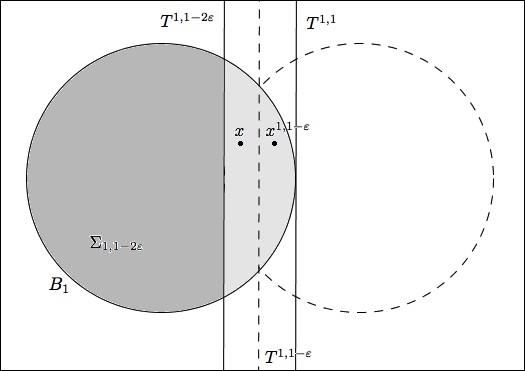}
\caption{Proof of \autoref{GGS4}, $N=2$, $i=1$.}
\label{sliding}
\end{figure}

The following Lemma is the starting point of the moving procedure. 

\begin{lemma}\label{GGS4}
Let $(u, v)\in C_0^{2\alpha}(\bar B_1) \times C_0^{2\beta}(\bar B_1)$ nontrivial solution to \eqref{problemfg}. There exists $\epsilon >0$ such that for all $\lambda \in [1- \epsilon, 1)$ one has
\begin{align}\label{eqnGGS4} 
\begin{split}
u(x)&>u(x^{i, \lambda}) \text{ for } x \in \Sigma_{i, \lambda}, \quad \frac{\partial u}{\partial x_i} <0 \text{ on }T_{i, \lambda} \cap B_1\\
v(x)&>v(x^{i, \lambda}) \text{ for } x \in \Sigma_{i, \lambda}, \quad \frac{\partial v}{\partial x_i} <0 \text{ on }T_{i, \lambda} \cap B_1
\end{split}
\end{align}
\end{lemma}
\begin{proof}
Since $x^{i,1} \in B_1^c$ for any $x \in \Sigma_{i,1}=B_1$, by \autoref{GGS3} there exists $\epsilon$ such that $ \frac{\partial u}{\partial x_i}<0$, $ \frac{\partial v}{\partial x_i}<0$ on $T_{i, l} \cap B_1$ for all $l \in (1-2\varepsilon, 1)$. 

Hence, for all $\lambda \in [1- \epsilon, 1)$ one has 
\[ u(x)>u(x^{i, \lambda}), v(x)>v(x^{i, \lambda}) \text{ for } x \in \Sigma_{i, \lambda}. \] 
Indeed, if $x_i \le 1-2\epsilon$ then $x^{i, \lambda} \in B_1^c$ and since $u>0$ in $B_1$ and $=0$ outside $B_1$, $u(x)>u(x^{i, \lambda})$. If $1-2\epsilon < x_i <1$, then two cases may occur: 
\begin{itemize}
\item $x^{i, \lambda} \in B_1^c$ and the conclusion follows as above,
\item $x^{i, \lambda} \in B_1$: in this case, it is enough to exploit the fact that  $\frac{\partial u}{\partial x_i}(x) <0$ for all $x \in T^{i, \lambda} \cap B_1$, with $\lambda \in (1-2\varepsilon, 1)$. 
\end{itemize}
Similarly for $v$. Therefore, \eqref{eqnGGS4} holds for all $\lambda \in [1-\epsilon, 1)$ and the proof is complete.
\end{proof}
 
We are now ready to slide the hyperplane to the critical position $\lambda=0$. 
\begin{proposition}\label{lemmovingplanes}
If $(u,v) \in C_0^{2\alpha}(\bar B_1) \times C_0^{2\beta}(\bar B_1)$ nontrivial solution to \eqref{problemfg} we have
\begin{multline}\label{eqmoving} 
\Lambda= \{ \lambda \in (0,1): u(x) > u(x^{i, \lambda}), v(x) > v(x^{i, \lambda}) \, \forall \, x \in \Sigma_{i, \lambda} \\
\frac{\partial u}{\partial x_i} <0, \frac{\partial u}{\partial x_i} <0 \text{ on }T_{i, \lambda} \cap B_1 \} = (0, 1). 
\end{multline}
\end{proposition}
\begin{proof}
By \autoref{GGS4} it turns out that $[ 1 - \epsilon, 1) \subset \Lambda$. Let $\bar{\lambda}$ be the smallest number such that $(\bar{\lambda}, 1) \subset \Lambda$. The proof will be complete once we show that $\bar{\lambda}=0$. By continuity one has 
\[ u(x) \ge u(x^{i, \bar{\lambda}}), v(x) \ge v(x^{i, \bar{\lambda}}) \text{ for all } x \in \Sigma_{i, \bar{\lambda}}. \]
By contradiction assume $\bar{\lambda}>0$. 
Take $x \in \Sigma_{i, \bar{\lambda}}$. Then
\begin{align*}
u(x) - u(x^{i, \bar{\lambda}})&=\int_{B_1} (G^{\alpha}(x, y)- G^{\alpha}(x^{i, \bar{\lambda}}, y))g(v(y)) \, dy \\
&= \int_{\Sigma_{i, \bar{\lambda}}} (G^{\alpha}(x, y)-G^{\alpha}(x^{i, \bar{\lambda}}, y)) g(v(y)) \, dy \\
&\quad+ \int_{\Sigma_{i, \bar{\lambda}}} (G^{\alpha}(x, y^{i, \bar{\lambda}})-G^{\alpha}(x^{i, \bar{\lambda}}, y^{i, \bar{\lambda}})) \tilde{g}(v(y^{i, \bar{\lambda}})) \, dy.
\end{align*}
By \autoref{GGS2}, two cases may occur: $g(v(y)) > \tilde{g}(v(y^{i, \bar{\lambda}}))$ for all $y \in \mathcal{O}^v_{i, \bar{\lambda}}$ or $\tilde{g}(v(y^{i, \bar{\lambda}}))>0$ for all $y \in \mathcal{O}^v_{i, \bar{\lambda}}$. In the first case, by \autoref{GGS1bis}
\begin{multline*} u(x) - u(x^{i, \bar{\lambda}}) > \int_{\Sigma_{i, \bar{\lambda}}} [G^{\alpha}(x, y) - G^{\alpha}(x^{i, \bar{\lambda}}, y) + \\ G^{\alpha}(x, y^{i, \bar{\lambda}}) - G^{\alpha}(x^{i, \bar{\lambda}}, y^{i, \bar{\lambda}})] \tilde{g}(v(y^{i, \bar{\lambda}})) \, dy \ge 0 \end{multline*}
whereas in the second case
\begin{multline*} u(x) - u(x^{i, \bar{\lambda}}) \ge \int_{\Sigma_{i, \bar{\lambda}}} [G^{\alpha}(x, y) - G^{\alpha}(x^{i, \bar{\lambda}}, y) \\+ G^{\alpha}(x, y^{i, \bar{\lambda}}) - G^{\alpha}(x^{i, \bar{\lambda}}, y^{i, \bar{\lambda}})]  \tilde{g}(v(y^{i, \bar{\lambda}})) \, dy > 0. \end{multline*}
Similar considerations hold for $v$. 
Hence 
\begin{equation}\label{terzopezzo} u(x) > u(x^{i, \bar{\lambda}}), v(x) > v(x^{i, \bar{\lambda}}) \text{ for all } x \in \Sigma_{i, \bar{\lambda}}. \end{equation} 
Due to \autoref{GGS3} there exists $\gamma_1$ such that 
\begin{equation}\label{primopezzo} \frac{\partial u}{\partial x_i} <0, \frac{\partial v}{\partial x_i} <0  \text{ on } T_{i, l} \cap B_1 \text{ for all }l \in (\bar{\lambda} - 2\gamma_1, \bar{\lambda}).\end{equation}
Now, by continuity, for any $x \in \bar{B}_1$ such that $x_i \le \bar{\lambda} - \gamma_1$ there exists $\gamma(x) >0$ such that
\[ u(x) \ge u(x^{i,l}),\, v(x) \ge v(x^{i, l}), \, l \in (\bar{\lambda}- \gamma(x), \bar{\lambda}] \]
and by compactness of $C= \{ x \in \bar{B}_1: x_i \le \bar{\lambda} - \gamma_1 \}$ one can take 
\[ \gamma=\min \{ \inf_C \gamma(x), \gamma_1 \}=\min \{ \min_C \gamma(x), \gamma_1 \}>0, \]
hence  for all $x \in \Sigma_{i, \bar{\lambda} - \gamma_1}$, 
\[ u(x) \ge u(x^{i,l}), \, v(x) \ge v(x^{i, l}), \, l \in (\bar{\lambda}- \gamma, \bar{\lambda}] \]
and exploiting the same argument as above, if $x \in \Sigma_{i, \bar{\lambda} - \gamma_1}$
 \begin{equation}\label{secondopezzo} u(x) > u(x^{i, l}),\, v(x) > v(x^{i,l}), \text{ for all } l \in (\bar{\lambda}- \gamma, \bar{\lambda}].\end{equation}
The conclusion follows in view of \eqref{terzopezzo}, \eqref{primopezzo} and \eqref{secondopezzo}.
\end{proof}

\begin{proof}[Proof of \autoref{propestimates}]

Follows by \autoref{lemmovingplanes}:  if $(u, v)$ is a solution to \eqref{problemfg}, then \eqref{eqmoving} holds true and since  \eqref{problemfg} is invariant by rotation and the domain is radially symmetric, this implies that $u, v$ are radially symmetric and $u'(r), v'(r) <0$. 
\end{proof}

\section{Proof of the main results}
\begin{proof}[Proof of \autoref{teorema2}]
The proof of \autoref{teorema2} is a simple combination of the previous steps.
Let us assume that 
\[ \mathcal{C} \cap (\{0\}\times C_0^{2\alpha}(\bar{B}_1) \times C_0^{2\beta}(\bar{B}_1)) = \{(0,0,0)\} \]
where $\mathcal{C}$ is defined as in \autoref{sequenzenonlimitate}.
Hence by \autoref{sequenzenonlimitate} we can find an unbounded sequence $(t_n, u_n, v_n)$  of solutions to \eqref{eqnt}. Note that $t_n >0$ and as a consequence $u_n >0$ and $v_n>0$ by \autoref{boggio}. However, for any fixed $n$, in view of 
\autoref{propestimates} with $f(u)=(t_n+\abs{u})^p$ and $g(v)=(t_n^{\vartheta}+\abs{v})^q$, we have that $u_n, v_n$ are radially symmetric and the global maxima are attained at $0$. Therefore, by \autoref{lemblowup} one concludes that there exists a nontrivial radial solution to \eqref{sysnonexistence}. 
\end{proof} 
\noindent It is now enough to combine \autoref{teorema2}, \autoref{GZL} and \autoref{GLZ2} to get a contradiction and hence to conclude the proof of \autoref{teorema}. As for \autoref{teoremabis}, one exploits \autoref{teorema2} and \autoref{nethradial}.

\begin{proof}[Proof of \autoref{coruniq}]

The proof relies on an idea originally due to Gidas, Ni and Nirenberg \cite{GidasNiNirenberg79}, see also \cite{CuiWangShi07, Dalmasso04, Dalmasso95}.
Let $(u, v)$ be a nontrivial solution to \eqref{21}. We know that $u, v$ are positive, radially symmetric and strictly decreasing. In particular, since the maximum is attained at $0$, we have $u'(0)=0$ and $v'(0)=0$. 
Moreover, 
\begin{equation}\label{(Delta u)'} 
r^{N-1}(\Delta u)'(r)=\int_0^r s^{N-1} (\Delta^2 u)(s) \, ds .  
\end{equation}
Indeed, this is straightforward by differentiating both sides of \eqref{(Delta u)'}.
As a consequence,
\begin{equation}\label{ce} 
(\Delta u)'(0)=\lim_{r \to 0} \frac{\int_0^r s^{N-1} (\Delta^2 u)(s) \, ds}{r^{N-1}}=0.
\end{equation}
Let $(w, z)$ be another nontrivial solution to \eqref{21} and let us set  
\[ \tilde{w}(r)=\lambda^s w(\lambda r) \text{ and } \tilde{z}(r)=\lambda^t z(\lambda r), \]
 where $s, t$ are chosen such that $(\tilde{w}, \tilde{z})$ satisfies
 \[ \begin{cases}
 \Delta^{2} \tilde{w}(r)=\tilde{z}^q(r) \, &0 \le r \le 1/\lambda \\
 -\Delta \tilde{z}(r)=\tilde{w}^p(r) \, &0 \le r \le 1/\lambda \\
\tilde{w}(1/\lambda)= \tilde{w}'(1/\lambda)=\tilde{z}(1/\lambda)=0, 
\end{cases} \]
namely $t=\frac{2+4p}{pq-1}$ and $s=\frac{2q+4}{pq-1}$, 
whereas $\lambda >0$ be such that 
 \begin{equation}\label{lambda}
\tilde{w}(0)= u(0).
 \end{equation}

\noindent \textbf{Claim:} 
\begin{equation}\label{claim}
\tilde{z}(0)=v(0) , \, \Delta \tilde w(0)=\Delta u(0). 
\end{equation}  

\noindent Let us suppose by contradiction for instance $(v- \tilde z )(0)>0$ and $\Delta( u- \tilde w)(0) >0$. Notice that by continuity $v-\tilde z >0$ on $[0, \delta)$ and $\Delta( u- \tilde w) >0$ on $[0, \varepsilon)$ for some $\delta, \varepsilon$ sufficiently small. Moreover $u- \tilde w >0$ on $(0, \varepsilon]$: indeed, if there exists $a \le \varepsilon$ such that $u(a)-\tilde w(a) \le 0$, then $\Delta(u-\tilde w)>0$ implies $u- \tilde w <0$ on $[0, a)$, which is a contradiction. 

\noindent Hence we can choose $R_1$ such that 
\begin{multline*} R_1= \sup \{ r \le \min \{1, 1/\lambda \} : \, (u- \tilde w)(s) >0, \, (v- \tilde z)(s) >0, \\
 \,( \Delta(u-\tilde w))(s) >0, \, s \in (0, r) \} .\end{multline*}
We have 
\begin{equation}\label{R_1}
(u- \tilde w)(R_1) >0, \, \Delta(u-\tilde w)(R_1)>0.
\end{equation}
Indeed, let us assume by contradiction that $(u-\tilde w)(R_1)=0$. Then, since $\Delta(u-\tilde w) >0$ on $[0, R_1)$ we have by the maximum principle $u-\tilde w <0$ on $[0, R_1)$. 
Analogously, if $\Delta(u-\tilde w)(R_1)=0$, then we have $\Delta^2(u-\tilde w)=v^q- \tilde z^q>0$ on $[0, R_1)$, thus $\Delta(u-\tilde w) <0$ on $(0, R_1)$. 
As a consequence, \eqref{R_1} holds. We distinguish two cases: $R_1<\min \{1, 1/\lambda \}$, and in this case $(v- \tilde z)(R_1)=0$, or $R_1=\min \{1, 1/\lambda \}$. 

\noindent In the first case, by applying the maximum principle to $-\Delta(v-\tilde z) =u^p- \tilde w^p>0$, one has $v-\tilde z <0$ on $(R_1, R_1+ \delta)$ for $\delta$ sufficiently small. 
We can set $R_2$ such that 
\begin{multline*} R_2= \sup \{ r \le \min \{1, 1/\lambda \} : \, (u- \tilde w)(s) >0, \, (v- \tilde z)(s) <0, \,  \\
\Delta(u-\tilde w)(s) >0, \, s \in (R_1, r) \} .\end{multline*}
As above, we have 
\[ (v- \tilde z)(R_2) <0, \, (u-\tilde w)(R_2)>0 \]
and moreover $R_2<\min \{1, 1/\lambda \}$, which implies $\Delta(u-\tilde w)(R_2)=0$, or $R_2=\min \{1, 1/\lambda \}$. Indeed, if $(v-\tilde z)(R_2)=0$, then by applying the maximum principle to $-\Delta(v-\tilde z) =u^p- \tilde w^p>0$ on $B_{R_2} \setminus \overline{B_{R_1}}$ we have $v- \tilde z >0$ on $(R_1, R_2)$; on the other hand, if $(u- \tilde w)(R_2)=0$, then $u- \tilde w <0$ on $[0, R_2)$, as $\Delta(u- \tilde w)>0$.  

\noindent We now apply iteratively the same reasoning as above to get a sequence (which can be finite or infinite) 
\[ 0 =R_0 < R_1 < R_2 < \dots \le \min \{1, 1/\lambda \} \]
such that
\[ u(R_{3k})=\tilde w(R_{3k}), \, v(R_{3k+1})=\tilde z(R_{3k+1}), \, \Delta u(R_{3k+2})=\Delta \tilde w(R_{3k+2}), \quad k \ge 0, \]
as long as $R_k < \min \{1, 1/\lambda \} $, see \autoref{table}.

\begin{table}\centering\caption{Sign of $u- \tilde w$, $v- \tilde z$, $\Delta(u- \tilde w)$. }
\label{table}\begin{tabular}{r|cccc}
&$(u-\tilde w)(s)$&$(v-\tilde z)(s)$&$\Delta(u- \tilde w)(s)$\\ \hline
$s=0$&=0&>0&>0\\ \hline
$s \in (0, R_1)$&>0&>0&>0\\ \hline
$s=R_1$&>0&=0&>0\\ \hline
$s \in (R_1, R_2)$&>0&<0&>0\\ \hline
$s=R_2$&>0&<0&=0\\ \hline
$s \in (R_2, R_3)$&>0&<0&<0\\ \hline
$s=R_3$&=0&<0&<0\\ \hline
$s \in (R_3, R_4)$&<0&<0&<0\\ \hline
$\vdots$ &$\vdots$ &$\vdots$ &$\vdots$ 
\end{tabular}
\end{table}

\noindent If $\{R_j\}$ is infinite, then we take the limit $R_*=\lim_{i \to \infty} R_i \le  \min \{ 1, 1/\lambda \}$ and by continuity and differentiability, it holds
\[ (u- \tilde w )(R_*)=0, \, (v- \tilde z)(R_*)=0, \, \Delta(u- \tilde w)(R_*)=0 \]
and
\[ (u'- \tilde w' )(R_*)=0, \, (v'- \tilde z')(R_*)=0, \, (\Delta(u- \tilde w))'(R_*)=0. \]
Next define
 \[ U(r)=(u(r), -\Delta u(r), v(r)) \quad 0 \le r \le 1 \]
 and
\[  W(r)=(\tilde{w}(r), -\Delta \tilde{w} (r), \tilde{z}(r)) \quad 0 \le r \le 1/\lambda. \]
For any $0 \le r \le R_*$ one has 
 \begin{equation}\label{UW_1} U(r)- W(r)=\int_{r}^{R_*} \frac{s}{N-2} \left( 1 - \left(\frac{s}{r} \right)^{N-2} \right)  (F(U(s))- F(W(s))) \, ds \end{equation}
where we set $F(x,y, z)=(y, z^q, x^p)$.
Since $p, q >1$, then $F$ is locally Lipschitz continuous, hence by the Gronwall Lemma, \eqref{UW_1} implies $U=W$ on $[0, R_*]$. 
This is in contradiction with the assumption $(v- \tilde z)(0)>0$. 

\noindent On the other hand, if the sequence stops at a maximum value $R_k$ then on $(R_{k-1}, R_k=\min \{1, 1/\lambda \} ]$ at least one of the following is verified: 
\begin{itemize}
\item $u- \tilde w$ and $v- \tilde z$ have opposite sign 
\item $u- \tilde w$ and $\Delta(u- \tilde w)$ have the same sign. 
\end{itemize}
Let for instance $u- \tilde w >0$ and $v- \tilde z \le 0$. Then

 \[ 0 \le (\tilde{z}-v)(\min (1, 1/\lambda))=
 \begin{cases}
-v(1/\lambda) & \text{ if } \lambda >1 \\
 0 & \text{ if } \lambda =1 \\
 \tilde{z}(1) & \text{ if } \lambda <1
 \end{cases} \]
and thus $\lambda \le 1$. On the other hand, 
 \[ 0> (\tilde{w}-u)(\min (1, 1/\lambda))=
 \begin{cases}
 -u(1/\lambda) & \text{ if } \lambda >1 \\
 0 & \text{ if } \lambda =1 \\
 \tilde{w}(1) & \text{ if } \lambda <1
 \end{cases} \]
thus we have $\lambda > 1$, which is a contradiction. 

\noindent Let now $u- \tilde w >0$ and $\Delta(u- \tilde{w}) \ge 0 $. Then, $(u - \tilde w)(\min \{1, 1/\lambda \}) > 0$ implies $\lambda >1$, whereas by the Hopf lemma $0< (u'- \tilde w' )(\min \{1, 1/\lambda \})=(u' - \tilde w')(1/\lambda)=u'(1/\lambda)<0$, a contradiction.

\noindent Therefore, since we have a contradiction in any case, we cannot have $(v- \tilde z )(0)>0$ and $\Delta( u- \tilde w)(0) >0$. In a similar fashion, one proves that also the other possible choices for the sign of $(v- \tilde z )(0)$ and $\Delta( u- \tilde w)(0)$ yield a contradiction, hence \eqref{claim} holds and the claim is proved. 

\noindent Finally, in view of \eqref{lambda} and \eqref{claim}, and since by \eqref{ce} $u'(0)=v'(0)=\tilde w'(0)= \tilde z'(0)=(\Delta u)'(0)=(\Delta \tilde w)'(0)=0$, for any $r \le \min \{ 1, 1/\lambda \}$
one has
 \begin{equation}\label{UW} U(r)- W(r)=\int_0^r \frac{s}{N-2} \left( 1 - \left(\frac{s}{r} \right)^{N-2} \right)  (F(W(s))- F(U(s))) \, ds \end{equation}
 where $F(x,y, z)=(y, z^q, x^p)$.
Since $p, q >1$, then $F$ is locally Lipschitz continuous, hence by the Gronwall Lemma, \eqref{UW} implies $U=W$ on $[0, \min \{1, 1/\lambda\}]$. 

\noindent Since $0<u(1/\lambda)=\tilde w(1/\lambda)=0$ if $\lambda >1 $, whereas $0=u(1)=\tilde w(1)>0$ if $\lambda <1 $, then $\lambda=1$ and hence $(u, v)=(w, z)$. 
\end{proof}

\appendix

\section{Appendix}
In what follows, we give a detailed proof of \autoref{GZL}. We refer to \cite{MitidieriPohozaev01} for the original proof for more generic higher order operators, see also \cite{CaristiDAmbrosioMitidieri08} where similar estimates are exploited to get representation formulas. In \cite{LiuGuoZhang06} the same result is proved for classical solutions.
Let us show the following preliminary lemma. 
\begin{lemma}\label{lemma}
Let $\rho=\abs{x}$, $s \ge 1$, $R >0$ and $h \colon \mathbb{R} \to \mathbb{R}$. Then 
\begin{equation} \label{Delta} \Delta^s h(\rho/R) = \sum_{i=1}^{2s} c_i \frac{h^{(i)}(\rho/R)}{R^i \rho^{2s-i}}
\end{equation}
for suitable coefficients $c_i \in \mathbb{R}$ not depending on $R$.
\end{lemma}
\begin{proof}
Note that 
\[ \Delta h(\rho/R)=\frac{h'' (\rho/R)}{R^2} + \frac{h'(\rho/R)}{\rho R} (N-1). \]
Let us assume by induction that \eqref{Delta} holds. Let us call 
\[ f(t)= \sum_{i=1}^{2s} c_i \frac{h^{(i)} (t)}{R^{2s} t^{2s-i}}. \]
One has 
\begin{multline*} f'(t) = \sum_{i=1}^{2s} c_i \frac{h^{(i)}(t)}{R^{2s} t^{2s-i+1}} + \sum_{i=1}^{2s} c_i \frac{h^{(i+1)}(t)}{t^{2s-i}R^{2s}} \\
= \sum_{i=1}^{2s} c_i \frac{h^{(i)}(t)}{R^{2s} t^{2s-i+1}} +\sum_{i=2}^{2s+1} c_{i-1} \frac{h^{(i)}(t)}{t^{2s+1-i} R^{2s}} =
\sum_{i=1}^{2s+1} c_i \frac{h^{(i)}(t)}{R^{2s} t^{2s-i+1}} \end{multline*}
where possibly different coefficients are always denoted by $c_i$.
Similarly,
\[ f''(t) = \sum_{i=1}^{2s + 2} c_i \frac{h^{(i)}(t)}{R^{2s} t^{2s+2-i}} \]
and therefore 
\[ f'(\rho/R)=\sum_{i=1}^{2s+1} c_i \frac{h^{(i)}(\rho/R)}{R^{i-1} \rho^{2s-i+1}} \quad f''(\rho/R) = \sum_{i=1}^{2s + 2} c_i \frac{h^{(i)}(\rho/R)}{R^{i-2} \rho^{2s+2-i}}. \]
As a consequence, 
\begin{multline*}
\Delta^{s+1} h(\rho/R)=\Delta f (\rho/R) =\frac{f'' (\rho/R)}{R^2} + \frac{f'(\rho/R)}{\rho R} (N-1)
\\  =\sum_{i=1}^{2s+2} c_i \frac{h^{(i)}(\rho/R)}{R^i \rho^{2s+2-i}}. \qedhere \end{multline*}
\end{proof}
\begin{remark}\label{rmkh}
Let us choose $h(s)= \psi^{\gamma}(s)$ where $\psi$ is a smooth and positive standard cut off function such that $\psi(s)=1$ for $0 \le s \le 1$ and $\psi(s)=0$ for $s >2$ and $\gamma>0$ fixed. Then for any $i = 1, \dots , N$ 
\begin{equation}\label{hi} h^{(i)} = \sum_{k_1+ \dots + k_i=i} c_{K} \psi^{(k_1)} \cdots \psi^{(k_i)} \psi^{\gamma - i } \end{equation}
where $c_K$ are real coefficients depending on $K=(k_1, \dots, k_i)$, $\gamma$ and $i$. 
Indeed, for $i=1$ one has $h'=\gamma \psi' \psi^{\gamma - 1}$. 
By induction let us assume \eqref{hi}; then 
\begin{align*}
h^{(i+1)} &=  \sum_{k_1+ \dots + k_i=i} c_{K} \psi^{(k_1+1)} \psi^{(k_2)}\cdots \psi^{(k_i)} \psi^{\gamma - i } + \dots \\
& \quad \qquad+  \sum_{k_1+ \dots + k_i=i} c_{K} \psi^{(k_1)} \cdots \psi^{(k_i +1)} \psi^{\gamma - i }  \\
& \quad \qquad+  \sum_{k_1+ \dots + k_i=i} c_{K} \psi^{(k_1)} \cdots \psi^{(k_i)} \psi^{\gamma - i -1} \psi' 
\end{align*}
and thus 
\begin{align*}
h^{(i+1)} &= \sum_{k_1+ \dots + k_i=i} c_{K} \psi^{(k_1+1)} \psi^{(k_2)}\cdots \psi^{(k_i)}\psi^{(0)} \psi^{\gamma - i -1} + \dots \\
& \quad \qquad+  \sum_{k_1+ \dots + k_i=i} c_{K} \psi^{(k_1)} \cdots \psi^{(k_i +1)} \psi^{(0)}\psi^{\gamma - i-1 }  \\
& \quad \qquad+  \sum_{k_1+ \dots + k_i=i} c_{K} \psi^{(k_1)} \cdots \psi^{(k_i)}\psi' \psi^{\gamma - i -1} \\
&=\sum_{k_1+ \dots + k_{i+1}=i+1} c_{K} \psi^{(k_1)} \psi^{(k_2)}\cdots \psi^{(k_{i+1})} \psi^{\gamma - i -1}.
\end{align*}
Hence by \autoref{lemma}
\begin{multline}\label{h}
\Delta^s h(\rho/R)=\sum_{i=1}^{2s} c_i \frac{h^{(i)}(\rho/R)}{R^i \rho^{2s-i}}\\
=\sum_{i=1}^{2s} c_i \frac{1}{R^i \rho^{2s-i}}\sum_{k_1+ \dots + k_i=i} c_{K} \psi^{(k_1)}(\rho/R) \cdots \psi^{(k_i)}(\rho/R) (\psi (\rho/R))^{\gamma - i }. \end{multline}
\end{remark}
\begin{proof}[Proof of \autoref{GZL}]
Let $(u,v) \in L^p_{loc} \times L^q_{loc}$ be a positive weak solution to \eqref{LaneEmden_mn}, namely 
\begin{equation}\label{weak}
\begin{cases}
\begin{aligned}
\int  u (-\Delta)^{\alpha} \varphi  \ge \int  \abs{v}^{q} \varphi  \\
\int v (-\Delta)^{\beta} \psi \ge \int \abs{u}^{p} \psi 
\end{aligned} & \mathbb{R^N}
\end{cases}
\end{equation}
where $\varphi, \psi \in C_0^{\infty}$. Take $\varphi \ge 0$ and call $p'=1-1/p$ the conjugate exponent of $p$, and $q'$ the conjugate exponent of $q$. Then, by \eqref{weak} and the H{\"o}lder inequality, 
\begin{multline}\label{maggiorazioni}
\int \abs{v}^q \varphi \le \int u (-\Delta)^{\alpha} \varphi \le \left( \int \abs{u}^p \varphi \right)^{1/p} \left( \int \frac{\abs{\Delta^{\alpha}\varphi}^{p'}}{\varphi^{p'-1}} \right)^{1/p'}  \\
\le \left( \int v (-\Delta)^{\beta} \varphi \right)^{1/p} \left(\int \frac{\abs{\Delta^{\alpha}\varphi}^{p'}}{\varphi^{p'-1}} \right)^{1/p'} \\
\le \left( \int \abs{v}^q \varphi \right)^{1/pq}  \left( \int \frac{\abs{\Delta^{\beta}\varphi}^{q'}}{\varphi^{q'-1}} \right)^{1/pq'}  \left( \int \frac{\abs{\Delta^{\alpha}\varphi}^{p'}}{\varphi^{p'-1}} \right)^{1/p'}   
\end{multline}
and hence 
\begin{multline}\label{maggcap}
\int \abs{v}^q \varphi \le \left( \int \frac{\abs{\Delta^{\beta}\varphi}^{q'}}{\varphi^{q'-1}} \right)^{\frac{q}{q'(pq-1)}} \left( \int \frac{\abs{\Delta^{\alpha}\varphi}^{p'}}{\varphi^{p'-1}} \right)^{\frac{pq}{p'(pq-1)}}   \\
=\left \{ cap_{\beta}(\varphi, q')^{q-1} cap_{\alpha}(\varphi, p')^{q(p-1)} \right \} ^{\frac{1}{pq-1}}
\end{multline}
where 
\[ cap_{\beta}(\varphi, r)= \int \frac{\abs{\Delta^{\beta}\varphi}^r}{\varphi^{r-1}}. \]

Choose $\varphi(x)=h(\rho/R)$ with $\rho=\abs{x}$ and $h$ as in \autoref{rmkh}.
If $\gamma > 2\beta q'$ then \eqref{h} yields 
\[ \begin{split}
cap&_{\beta}(\varphi, q')= \int \frac{\abs{\Delta^{\beta} \varphi(x)}^{q'}}{(\varphi(x))^{(q'-1)}} \, dx \\
&\le C \int \sum_{j=1}^{2\beta} \sum_{k_1+ \dots + k_j=j}  \frac{\abs{\psi^{(k_1)}(\rho/R) \dots \psi^{(k_j)}(\rho/R)}^{q'}}{R^{jq'}\rho^{q'(2\beta - j)}} \frac{(\psi(\rho/R))^{q'(\gamma - j)}}{(\psi(\rho/R))^{\gamma(q'-1)}} \, dx \\
&= C \int_{\bar{B}_{2R} \setminus B_{R}} \sum_{j=1}^{2\beta} \sum_{k_1+ \dots + k_j=j}  (\psi(\rho/R))^{\gamma-q' j} \frac{\abs{\psi^{(k_1)}(\rho/R) \dots \psi^{(k_j)}(\rho/R)}^{q'}}{R^{jq'}\rho^{q'(2\beta - j)}} \, dx \\
&\le C \int_{\bar{B}_{2R} \setminus B_{R}} \sum_{j=1}^{2\beta} \sum_{k_1+ \dots + k_j=j}  (\psi(\rho/R))^{\gamma-q' j} \frac{\abs{\psi^{(k_1)}(\rho/R) \dots \psi^{(k_j)}(\rho/R)}^{q'}}{R^{2\beta q'}} \, dx.
\end{split}
\]
Hence, by performing the change of variable $x/R \mapsto y$, 
\[
\begin{split}
cap_{\beta}(\varphi, q') &\le C R^N \sum_{\substack{k_1+ \dots + k_j=j \\ j=1, \dots 2 \beta}} \int_{\bar{B}_{2}\setminus B_1} (\psi(\abs{y}) )^{\gamma - q' j}\frac{\abs{\psi^{(k_1)}(\abs{y}) \dots \psi^{(k_j)}(\abs{y})}^{q'}}{R^{2\beta q'}} \, dy \\
&\le \hat C \frac{1}{R^{2\beta q' - N}}.
\end{split}
\]
The last inequality holds true since $\gamma > q' j$ for any $j \le 2 \beta$ and hence \[ (\psi(\abs{y}))^{\gamma - q' j} \abs{\psi^{(k_1)}(\abs{y}) \dots \psi^{(k_j)}(\abs{y})}^{q'}\] is a $C^{\infty}$ function, therefore it is integrable on the annulus $\bar{B}_2 \setminus B_1$; moreover, for $j=1$, it equals to $(\psi(\abs{y}))^{\gamma - q'} \abs{\psi'(\abs{y})} ^{q'}$, which is $\ne 0$ and $\ge 0$ and as a consequence $\hat C \ne 0$.
In a similar fashion one gets that if $\gamma > 2\alpha p'$ then 
\[ cap_{\alpha}(\varphi, p') \le C \frac{1}{R^{2\alpha p' - N}} \]
 where $C$ does not depend on $R$. Therefore by \eqref{maggcap} 
\begin{equation}\label{stimav} \int \abs{v}^q \varphi \le C R^{-(\frac{2\beta q + N + 2\alpha pq - N pq}{pq-1})}. \end{equation}
Hence, if 
\[ 2\beta q + N + 2\alpha pq - N pq>0, \]
by letting $R \to \infty$, one has $u=v=0$. 
Analogously  
\[ \int \abs{u}^p \varphi \le C R^{-(\frac{2\alpha p + N + 2\beta pq - N pq}{pq-1})} \]
and if 
\[ 2\alpha p + N + 2\beta pq - N pq >0 \]
then $u=v=0$. 

Assume now that $p, q$ are on the critical curve, namely
\[ 2\beta q + N + 2\alpha pq - N pq=0 \]
or
\[ 2\alpha p + N + 2\beta pq - N pq =0. \]
Let us consider the first case (the second one is similar): by \eqref{stimav} $v$ is in $L^q$; moreover, 
\[ \int_{B_R} \abs{v}^q \varphi \le \int_{B_R} u (-\Delta)^{\beta} \varphi = 0 \]
and
\begin{multline*} 
\int \abs{v}^q \varphi = \int_{B_R} \abs{v}^q \varphi + \int_{R \le \abs{x} \le 2 R} \abs{v}^q \varphi + \int_{\abs{x} > 2 R} \abs{v}^q \varphi \\
= \int_{R \le \abs{x} \le 2 R} \abs{v}^q \varphi \le  \int_{R \le \abs{x} \le 2 R} \abs{v}^q
\end{multline*}
with $\varphi$ as above. Therefore, since 
\[ \left \{ cap_{\beta}(\varphi, q')^{q-1} cap_{\alpha}(\varphi, p')^{q(p-1)} \right \} ^{\frac{1}{pq-1}} \le C R^{-(\frac{2\beta q + N + 2\alpha pq - N pq}{pq-1})} = C,   \]
by \eqref{maggiorazioni} one has 
\begin{equation}\label{stima2} \int \abs{v}^q \varphi \le C \left( \int_{R \le \abs{x} \le 2 R} \abs{v}^q \right)^{1/pq}=C \left( \int \abs{v}^q \chi_{\{R \le \abs{x} \le 2 R\}} \right)^{1/pq}.
\end{equation}
Hence, by letting $R \to \infty$ in \eqref{stima2}, since $v$ is in $L^q$, by the dominated convergence theorem
one has 
\[ \int \abs{v}^q \le 0\]
and thus $v=0$ which in turn implies $u=0$.
\end{proof}
\begin{remark}
Let us consider the case $N \le 2\alpha$ or $N \le 2\beta$.
If $p,q >1$, then there exists no weak nontrivial solution to \eqref{LaneEmden_mn}.
Indeed, the estimates in the proof above do not depend on the choice of $N$, therefore if 
\[ 2\beta q + N + 2\alpha pq - N pq>0 \]
or
\[ 2\alpha p + N + 2\beta pq - N pq >0 \]
then the unique solution is $0$. However, if $N \le 2\alpha$ then 
\[ 2\beta q  + N + 2\alpha pq - Npq \ge 2 \beta q + N + 2\alpha pq - 2 \alpha pq= 2\beta q + N > 0, \]
whereas if $N \le 2\beta$
\[ 2\alpha p + N + 2\beta pq - N pq >0 \]
and the conclusion follows.
\end{remark}

\section*{Acknowledgements}
The author wishes to thank prof.~Daniele Cassani for the careful reading of this paper and for many useful comments and suggestions. Many thanks go also to the referee for her/his valuable remarks.

\printbibliography

\end{document}